\pdfoutput=1
\documentclass[3p]{elsarticle}
\usepackage{amssymb,amsmath,amsthm,graphicx,tkz-graph,bm,adjustbox,microtype,thmtools,thm-restate}
\usepackage[labelfont=bf,labelsep=period]{caption}
\usetikzlibrary{scopes,calc}

\theoremstyle{plain}
\newtheorem{theorem}{Theorem}
\newtheorem{lemma}{Lemma}
\newtheorem{corollary}{Corollary}
\newtheorem{proposition}{Proposition}
\newtheorem{problem}{Problem}
\newtheorem{conjecture}{Conjecture}

\theoremstyle{definition}
\newtheorem{definition}{Definition}

\DeclareMathOperator{\dB}{dB}
\DeclareMathOperator{\GF}{GF}
\DeclareMathOperator{\lcm}{lcm}

\usepackage[pdftex,colorlinks]{hyperref}

\begin{document}
\begin{frontmatter}
	\title{Partitioning de~Bruijn Graphs into Fixed-Length Cycles\\ for Robot Identification and Tracking}
	\author[monash ecse,monash maths]{Tony~Grubman\corref{cor1}}
	\ead{tony.grubman@monash.edu}
	\author[monash ecse,utc heudiasyc]{Y.~Ahmet~\c{S}ekercio\u{g}lu}
	\ead{asekerci@utc.fr}
	\author[monash maths]{David~R.~Wood\fnref{fn1}}
	\ead{david.wood@monash.edu}
	
	\cortext[cor1]{Corresponding author.}
	\fntext[fn1]{Research supported by the Australian Research Council.}
	
	\address[monash ecse]{Department of Electrical and Computer Systems Engineering, Monash University, Melbourne, Australia}
	\address[monash maths]{School of Mathematical Sciences, Monash University, Melbourne, Australia}
	\address[utc heudiasyc]{Heudiasyc Laboratory, Compi{\`{e}}gne University of Technology, Compi{\`{e}}gne, France}
	\begin{abstract}
		We propose a new camera-based method of robot identification, tracking and orientation estimation. The system utilises coloured lights mounted in a circle around each robot to create unique colour sequences that are observed by a camera. The number of robots that can be uniquely identified is limited by the number of colours available, $q$, the number of lights on each robot, $k$, and the number of consecutive lights the camera can see, $\ell$. For a given set of parameters, we would like to maximise the number of robots that we can use. We model this as a combinatorial problem and show that it is equivalent to finding the maximum number of disjoint $k$-cycles in the de~Bruijn graph $\dB(q,\ell)$.
		
		We provide several existence results that give the maximum number of cycles in $\dB(q,\ell)$ in various cases. For example, we give an optimal solution when $k=q^{\ell-1}$. Another construction yields many cycles in larger de~Bruijn graphs using cycles from smaller de~Bruijn graphs: if $\dB(q,\ell)$ can be partitioned into $k$-cycles, then $\dB(q,t\ell)$ can be partitioned into $tk$-cycles for any divisor $t$ of $k$. The methods used are based on finite field algebra and the combinatorics of words.
	\end{abstract}
	\begin{keyword}
		graph theory \sep
		robot network \sep
		de~Bruijn graph \sep
		graph decomposition \sep
		pose estimation \sep
		linear feedback shift register
	\end{keyword}
\end{frontmatter}

\section{Introduction}
A \emph{robot network} is a collection of robots working together to achieve a common goal. In order for the robots in such a network to cooperate effectively, the ability to observe each other's movements is critical. In many applications, distinguishing between the robots is necessary, but is usually difficult because the robots are identical.

For example, in a \emph{formation control} system, robots collectively arrange themselves in some fixed geometric configuration \cite{anderson2008,das2002}. Each robot controls its position relative to its neighbours. To achieve this, the robot must continuously measure the position and determine the identity of each neighbour. Some formation control systems may also benefit from knowledge of the relative orientation of its neighbours, since this information can be used to coordinate views and improve the stability of the system.

We present a novel camera-based method for robot identification, orientation estimation, and approximate distance/angle measurements. The system uses a camera to observe sequences of coloured lights mounted on the robots. The lights are mounted in a circle around each robot (in a plane parallel to the ground), such that a camera may see only some of the lights. The sequences of colours are chosen so that any consecutive subsequence of sufficient length corresponds uniquely to a particular robot in a particular orientation.

This system was implemented in an existing network of \emph{eBugs}. The \emph{eBug} \cite{dademo2011} is a robotics platform designed at Monash University's Wireless Sensor and Robot Networks Laboratory \cite{wsrnlab}. It is equipped with sixteen RGB LEDs (red, green and blue light-emitting diodes) on its perimeter, which can be programmed to display a sequence of colours. A photo of an eBug may be seen in Figure~\ref{fig:ebug}.

\begin{figure}[htb]
\centering
\includegraphics[width=0.45\linewidth]{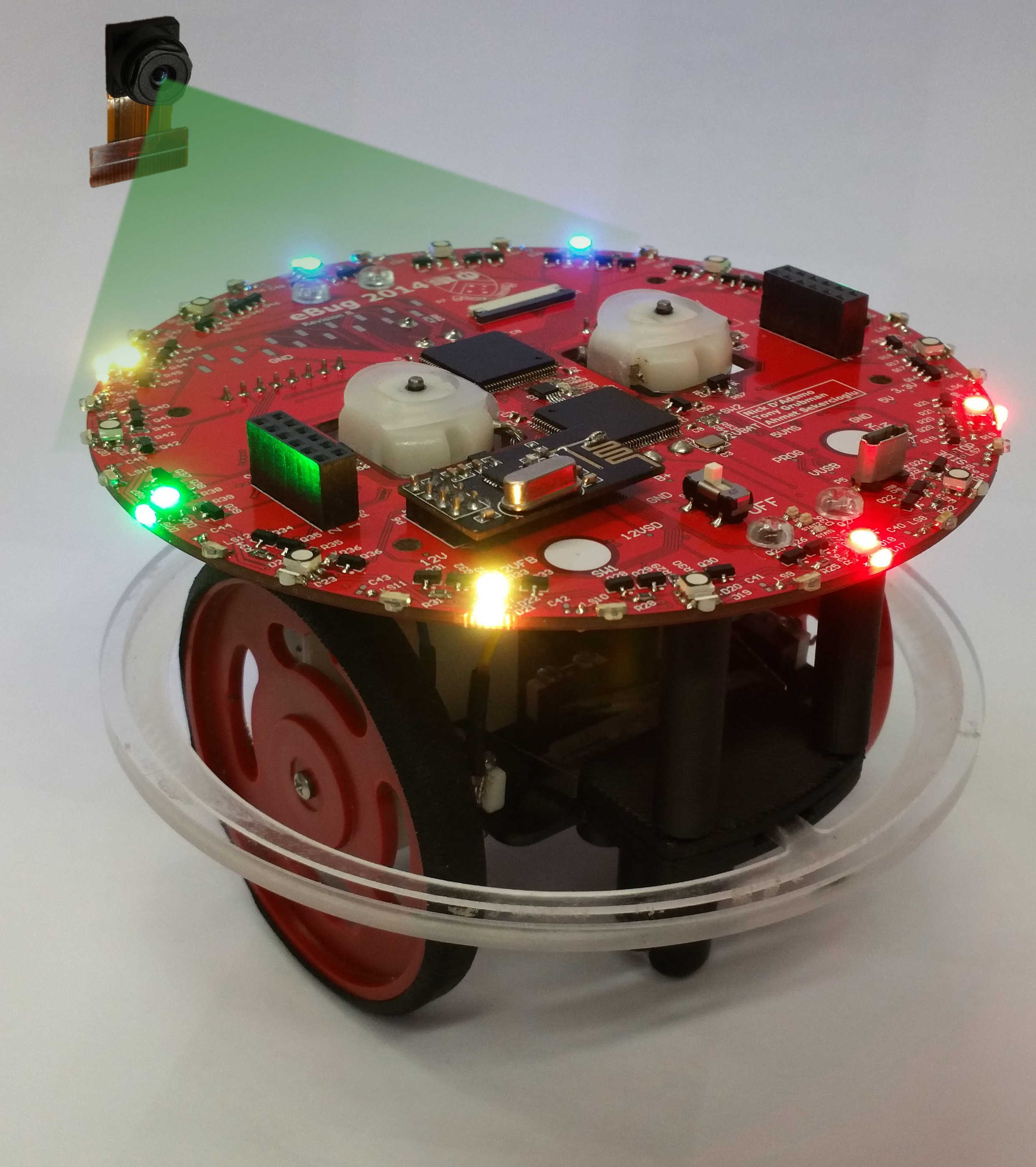}
\caption{An eBug displaying a colour sequence on eight of its sixteen LEDs. The camera viewing the eBug can only see three of these LEDs.}
\label{fig:ebug}
\end{figure}

Figure~\ref{fig:example colouring} shows an example of a colouring of four eBugs. These eBugs have only eight LEDs, and use only two different colours for illustrative purposes. Any subsequence of five LEDs is coloured with a unique pattern. For example, the sequence \tikz[baseline=-0.5ex]{\colorlet{a}{yellow!100!white}\colorlet{b}{blue!50!black}\foreach\colour[count=\i,evaluate=\colour using \colour*100] in {0,0,1,0,1} \node[draw,circle,shading=ball,ball color=a!\colour!b,inner sep=0.08cm] at (\i*0.3,0) {};} appears (counter-clockwise) only on the right side of the second eBug.

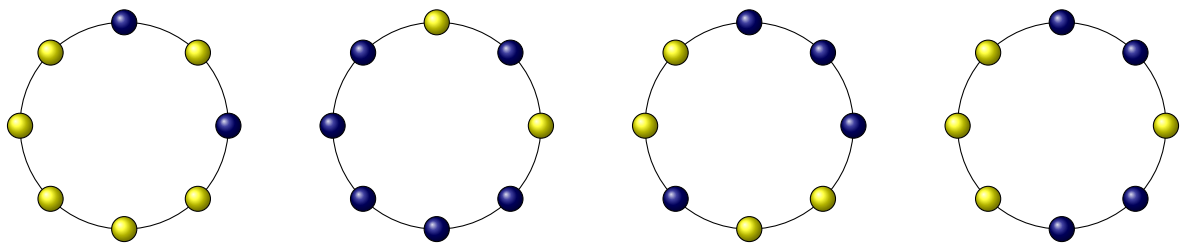
\begin{figure}[htb]
\centering
\begin{tikzpicture}[scale=1.5*\linewidth/18cm]
\colorlet{a}{yellow!100!white}
\colorlet{b}{blue!50!black}
\foreach \x/\seq in {
	0/{1,0,1,1,1,1,1,0},
	1/{0,1,0,0,0,0,0,1},
	2/{0,0,1,1,0,1,1,0},
	3/{0,0,1,1,1,0,0,1}}
	{
		\begin{scope}[xshift=3cm*\x]
		\draw (0,0) circle [radius=1cm];
		\foreach\colour[count=\i,evaluate=\colour using \colour*100] in \seq
			\node[draw,circle,shading=ball,ball color=a!\colour!b,radius=0.1cm] at ({360/8*\i}:1cm) {};
		\end{scope}
	}
\end{tikzpicture}
\caption{Example eBug colouring with $q=2$, $k=8$ and $\ell=5$.}\label{fig:example colouring}
\end{figure}

In a real system, there are limits on the number of colours a camera may reliably distinguish. Similarly, spatial resolution of the camera limits the number of detectable LEDs around each eBug. Therefore, for a given set of parameters, we want to maximise the number of eBugs that we can use in the system. This maximum, the \emph{eBug number}, is formally defined below.

\begin{definition}[eBug number]\label{def:ebug number}
Suppose every eBug has $k$ LEDs, each of which can be illuminated in one of $q$ colours, and suppose further that a camera can reliably detect $\ell\leq k$ consecutive LEDs. An assignment of colours to the LEDs of all eBugs is \emph{$\ell$-valid} if the camera can distinguish each eBug in each of the $k$ orientations. The \emph{eBug number} $\mathcal{E}(q,k,\ell)$ is the maximum number of eBugs for which there exists an $\ell$-valid assignment of colours.
\end{definition}

As well as modelling an actual problem that arises in robot networks, determining eBug numbers is a natural combinatorial problem of independent interest.

Each of the $q^\ell$ possible sequences the camera can see cannot appear more than once, and each eBug uses $k$ distinct sequences. This gives the following upper bound for the eBug number:

\begin{equation}\label{eq:E upper bound}
\mathcal{E}(q,k,\ell)\leq\frac{q^\ell}{k}.
\end{equation}

Colourings that achieve the upper bound in \eqref{eq:E upper bound} are called \emph{optimal}. In such colourings, each sequence of $\ell$ colours appears on some eBug. Note that when $k=\ell$, no $\ell$-sequence of a constant colour can appear on an eBug since all orientations would be identical. Thus optimal colourings can only exist for $k>\ell$.

A lower bound may be obtained by applying the Lov\'{a}sz local lemma \cite{erdos1975,spencer1977}: consider a random colouring of $n$ eBugs, with each of the $nk$ LEDs coloured independently and uniformly at random. For each pair $(i,j)$ of LED sequences (of length $\ell$), let $A_{ij}$ be the event that the same colour sequence has been assigned to $i$ and $j$. Thus the colouring is $\ell$-valid exactly when none of the events $A_{ij}$ occurs. Since there are exactly $nk$ LED sequences, and each sequence overlaps with at most $2\ell-1$ other sequences, each event $A_{ij}$ depends on at most $nk(2\ell-1)-1$ other events. The probability of each $A_{ij}$ is at most $q^{-\ell}$ (less if $i$ and $j$ are overlapping). Therefore, by the local lemma, there is an $\ell$-valid colouring whenever $eq^{-\ell}nk(2\ell-1)\leq 1$, where $e$ is Euler's number. Hence we obtain the following lower bound:

\[\mathcal{E}(q,k,\ell)\geq\left\lfloor\frac{q^\ell}{(2\ell-1)ek}\right\rfloor.\]

For a fixed value of $\ell$, this bound is within a constant factor of the upper bound in \eqref{eq:E upper bound}. In actual camera systems, however, it is reasonable to assume that $\ell$ is proportional to $k$, since a camera can usually detect a fixed arc of the LED circle. Thus the lower bound is rather crude, and ultimately we would like to solve the following problem.

\begin{problem}\label{prob:determine E}
Determine $\mathcal{E}(q,k,\ell)$ exactly.
\end{problem}

For small values of $q$ and $\ell$, a computer search was performed to find large $\ell$-valid colourings. Surprisingly, optimal colourings were found in many cases. These experiments confirm the following conjecture for all $q$ and $\ell$ with $q^\ell\leq 81$. While Problem~\ref{prob:determine E} is likely to be very difficult to solve in general, a mathematically interesting problem is to characterise when optimal colourings exist (hopefully by proving Conjecture~\ref{conj:E always optimal}).

\begin{conjecture}\label{conj:E always optimal}
$\mathcal{E}(q,k,\ell)=\dfrac{q^\ell}{k}$ whenever $k$ divides $q^\ell$ and $k>\ell$.
\end{conjecture}

This paper provides constructions for some infinite families of optimal colourings, and as such gives evidence to support this conjecture.

In Section~\ref{sec:preliminaries}, Problem~\ref{prob:determine E} is shown to be equivalent to finding many cycles in a \emph{de~Bruijn graph}, with Conjecture~\ref{conj:E always optimal} corresponding to a partition into cycles (see Proposition~\ref{prop:partition conditions}). Existing results about \emph{de~Bruijn sequences} are also discussed.

A well-known algebraic construction of de~Bruijn sequences is given in Section~\ref{sec:LFSRs}; we extend this construction to prove some existence results for eBug colourings. The major result of this section is Theorem~\ref{thm:lfsr translation}, which proves Conjecture~\ref{conj:E always optimal} for infinitely many values.

\begin{restatable}{theorem}{lfsrtranslation}\label{thm:lfsr translation}
	If $q$ is a prime power and $\ell\geq 1$, then $\mathcal{E}(q,q^{\ell-1},\ell)=q$.
\end{restatable}

In Section~\ref{sec:necklaces}, we introduce \emph{necklaces} and how they relate to de~Bruijn graphs. We then prove the following theorems, both of which yield large eBug colourings from smaller ones.

\begin{restatable}{theorem}{productcolouring}\label{thm:product colouring}
	Fix a value of $\ell$ and set $\mathcal{E}_1=\mathcal{E}(q_1,k_1,\ell)$ and $\mathcal{E}_2=\mathcal{E}(q_2,k_2,\ell)$. Then \[\mathcal{E}(q_1q_2,\lcm(k_1,k_2),\ell)\geq\gcd(k_1,k_2)\,\mathcal{E}_1\,\mathcal{E}_2.\]
\end{restatable}

\begin{restatable}{theorem}{interleaving}\label{thm:interleaving}
	$\mathcal{E}(q,tk,t\ell)\geq\frac{k^{t-1}}{t}\mathcal{E}(q,k,\ell)^t$ whenever $t$ divides $k$.
\end{restatable}

These theorems preserve optimality, so we may use them to find optimal colourings for large numbers of eBugs.

\section{Preliminaries}\label{sec:preliminaries}
\subsection{de~Bruijn graphs}
A valid colouring of eBugs has an interesting interpretation as cycles in a de~Bruijn graph. These graphs were discovered independently by de~Bruijn \cite{debruijn1946} and Good \cite{good1946} in 1946.

\begin{definition}
The $\ell$-th order $q$-ary \emph{de~Bruijn graph} $\dB(q,\ell)$ is the digraph $(V,E)$, where $V=\mathbb{Z}_q^\ell$ and $E=\{(a_0a_1\dots a_{\ell-1},a_1a_2\dots a_\ell)\mid a_i\in \mathbb{Z}_q\}$.
\end{definition}

The vertices of $\dB(q,\ell)$ are words of length $\ell$ over an alphabet of size $q$. There is an edge from $u$ to $v$ if shifting $u$ left and appending any letter gives $v$. An example of such a graph is shown in Figure~\ref{fig:example dB}.

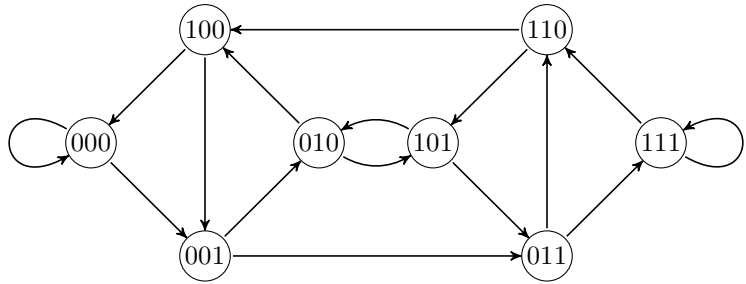
\begin{figure}[htb]
\centering
\begin{tikzpicture}[scale=1]
\SetVertexMath
\tikzset{VertexStyle/.style={shape=circle,fill=white,inner sep=1pt,outer sep=0pt,minimum size=0pt,draw}}
\SetGraphUnit{1.5}
\begin{scope}[xshift=-1.5*\GraphUnit cm]
	\Vertices{circle}{010,100,000,001}
\end{scope}
\begin{scope}[xshift=1.5*\GraphUnit cm]
	\Vertices{circle}{111,110,101,011}
\end{scope}
\Edges[style={post}](000,001,010,100,001,011,111,110,101,011,110,100,000)
\Edges[style={post,bend right}](101,010,101)
\Loop[dist=\GraphUnit*0.8cm,style={out=150,in=-150,thick,post}](000)
\Loop[dist=\GraphUnit*0.8cm,style={out=-30,in=30,thick,post}](111)
\end{tikzpicture}
\caption{Example de~Bruijn graph --- $\dB(2,3)$.}\label{fig:example dB}
\end{figure}

There is also an alternative, equivalent definition of de~Bruijn graphs that involves iteratively taking line digraphs \cite{zhang1987}. In this construction, $\dB(q,1)$ is defined as the complete digraph on $q$ vertices with loops. Higher order de~Bruijn graphs are defined as follows: $\dB(q,\ell+1)$ is the line digraph of $\dB(q,\ell)$. The vertices in $\dB(q,\ell+1)$ correspond to edges in $\dB(q,\ell)$. Note that while cycles in $\dB(q,\ell)$ map directly to cycles in $\dB(q,\ell+1)$, the converse is not always true: there may be repeated vertices when a cycle from $\dB(q,\ell+1)$ is projected down to $\dB(q,\ell)$. The objects in $\dB(q,\ell)$ that correspond to cycles in $\dB(q,\ell+1)$ are called \emph{circuits}, which are closed walks with no repeated edges (vertex repetition is allowed).

\begin{proposition}\label{prop:partition conditions}
The following are equivalent:
\begin{enumerate}
\item $\mathcal{E}(q,k,\ell)=\dfrac{q^\ell}{k}$.
\item There is a partition of the vertex set of $\dB(q,\ell)$ into pairwise disjoint $k$-cycles.
\item There is a partition of the edge set of $\dB(q,\ell-1)$ into pairwise edge-disjoint $k$-circuits.
\end{enumerate}
\end{proposition}

\begin{proof}
($1\Longleftrightarrow 2$)
Suppose that each vertex of $\dB(q,\ell)$ corresponds to a particular camera view of $\ell$ consecutive LEDs on some eBug. Rotating the eBug to the left corresponds to following an edge in the graph, since the LEDs shift to the right and one new LED is visible. Hence a cycle of length $k$ in $\dB(q,\ell)$ corresponds to the colouring of a single eBug with $k$ LEDs. A set of multiple disjoint cycles gives an $\ell$-valid colouring of multiple eBugs (because vertices are not repeated, each orientation is uniquely identifiable), so the eBug number $\mathcal{E}(q,k,\ell)$ equals the maximum number of \emph{disjoint $k$-cycles} in $\dB(q,\ell)$. If every vertex is in one of the $k$-cycles, then each colour sequence appears on some eBug. Conversely, if any given colour sequence can be found on some eBug, then the corresponding vertex is in one of the $k$-cycles. Thus optimal colourings exist exactly when the whole graph can be partitioned into disjoint $k$-cycles.

($2\Longleftrightarrow 3$) The equivalence follows immediately from the line digraph construction.
\end{proof}

Bryant studied edge decompositions of complete directed graphs with loops \cite{bryant1994}, which correspond to the first order de~Bruijn graphs $\dB(q,1)$. The main result of \cite{bryant1994} was that $\dB(q,1)$ can be decomposed into $k$-circuits if and only if $k\geq 3$ and $k$ divides $q^2$. By Proposition~\ref{prop:partition conditions}, this solves Conjecture~\ref{conj:E always optimal} for $\ell=2$.

A similar problem was also posed by Dudeney in 1917 \cite{dudeney1917}, now commonly known as ``Dudeney's round table problem''. This problem is equivalent to finding a set of Hamiltonian cycles in the complete graph $K_n$, such that every path of two edges appears in exactly one of the cycles. Dudeney's problem was solved for even $n$ \cite{kobayashi1993}, and also some other cases such as when $n-1$ is a prime power \cite{nakamura1980}. A generalisation of Dudeney's problem was studied in \cite{kobayashi2002}; here $K_n$ is covered by $k$-cycles with the same property (instead of $n$-cycles). The main difference between these problems and Problem~\ref{prob:determine E} is that we are concerned with directed circuits in digraphs, and that we allow loops on the vertices.

There is a body of research on cycle decompositions of complete graphs (see \cite{bryant2007} for an introduction and \cite{bryant2014} for recent results), and also some work relating to decompositions into fixed-length directed cycles \cite{alspach2003}. The methods used, however, are very specific to the special structure of complete graphs, and cannot be applied to de~Bruijn graphs. There are also results about decomposing de~Bruijn graphs into \emph{variable-length} cycles, using techniques like splitting and merging existing cycles \cite{cohn1972}. Golomb's conjecture, which was proven by Mykkeltviet \cite{mykkeltveit1972}, states that the decomposition of binary de Bruijn graphs into the largest number of disjoint cycles is the decomposition into \emph{necklaces} (see Section~\ref{sec:necklaces} for a definition). These results, unfortunately, cannot easily be applied to help with Conjecture~\ref{conj:E always optimal}, since the specific requirement of fixed-length cycles is quite restrictive.

\subsection{de~Bruijn sequences}
Note that in the de~Bruijn graph $\dB(q,\ell)$, every vertex has in-degree and out-degree $q$. Also, a path can be found from any vertex $u$ to any vertex $v$ by shifting in letters of $v$ one at a time, so the graph is connected. Hence $\dB(q,\ell)$ is Eulerian, and the next de~Bruijn graph $\dB(q,\ell+1)$ is Hamiltonian (since an Eulerian circuit in $\dB(q,\ell)$ is equivalent to a Hamiltonian cycle in $\dB(q,\ell+1)$). This simple fact gives us a starting point for Conjecture~\ref{conj:E always optimal} in the $k=q^\ell$ case: it shows that $\mathcal{E}(q,q^\ell,\ell)=1$ for every $q$ and $\ell$.

Hamiltonian cycles in de~Bruijn graphs are called \emph{de~Bruijn sequences}. The number of $(q,\ell)$-de~Bruijn sequences is
\[\dfrac{(q!)^{q^{\ell-1}}}{q^\ell}.\]

This result is due to van~Aardenne-Ehrenfest and de~Bruijn \cite{vanaardenne-ehrenfest1951}, and uses an equivalence between spanning arborescences and Eulerian circuits in Eulerian digraphs.

There are several known methods for generating de~Bruijn sequences. One construction \cite[\textsection 7.3]{ruskey2003} gives the lexicographically smallest sequence for any given values of $q$ and $\ell$ through clever concatenation of necklaces. This method is described in Section~\ref{sec:necklaces intro}. Another construction involves calculations in finite fields \cite[\textsection 7.7]{ruskey2003}. This only works when $q$ is a prime power, but has a very simple implementation, which is described in Section~\ref{sec:dB sequence from LFSR}.

\section{Linear feedback shift registers}\label{sec:LFSRs}
In this section, algebraic properties of finite fields are exploited to find interesting structures in de~Bruijn graphs. Section~\ref{sec:dB sequence from LFSR} describes a well-known construction of de~Bruijn sequences; we extend this construction further in Sections~\ref{sec:lfsr splitting} and \ref{sec:lfsr translations} to find multiple $k$-cycles in de~Bruijn graphs. We assume that the reader is familiar with elementary group and field theory; see \cite{fraleigh2003} for example.

\subsection{Construction}\label{sec:dB sequence from LFSR}
Let $q$ be a prime power, and choose a primitive element $\alpha$ from the finite field $F:=\GF(q^\ell)$. That is, $\alpha$ generates the multiplicative group $F^*=F\setminus\{0\}$. We may consider $F$ to be an $\ell$-dimensional vector space over $\GF(q)$, in which case $\{1,\alpha,\alpha^2,\dots,\alpha^{\ell-1}\}$ is a basis. In particular, $\alpha^\ell$ can be written as a linear combination of these basis vectors: $\alpha^\ell=p_0+p_1\alpha+\dots+p_{\ell-1}\alpha^{\ell-1}$ (this is called the \emph{minimal polynomial} of $\alpha$ over $\GF(q)$). 

A \emph{linear feedback shift register} (LFSR) is a digital circuit that generates elements of $F^*$ by successive multiplication by $\alpha$. The simplest implementation, a \emph{Galois LFSR}, represents the field elements as vectors in $\GF(q)^\ell$ with respect to the basis $\{1,\alpha,\alpha^2,\dots,\alpha^{\ell-1}\}$. Multiplication of a vector $\bm{v}:=(v_0,v_1\dots,v_{\ell-1})$ by $\alpha$ is simply a shift of the vector to the right, except that an $\alpha^\ell$ term is produced. But $\alpha^\ell$ can be rewritten in terms of the basis vectors, so the multiplication corresponds to the function $\bm{v}\mapsto(0,v_0,v_1,\dots,v_{\ell-2})+v_{\ell-1}(p_0,p_1,\dots,p_{\ell-1})$. Since the new state is a linear transformation of the previous state, this function can be expressed as the matrix equation in $\bm v\mapsto M\bm v$ (over $\GF(q)$), where the \emph{state change matrix}\footnote{In linear algebra, this matrix is also known as the \emph{companion matrix} for the minimal polynomial of $\alpha$.} $M$ is given by

\[M=\begin{pmatrix}
0& 0&\cdots& 0& p_0\\
1& 0&\cdots& 0& p_1\\
0& 1&\cdots& 0& p_2\\
\vdots&\vdots&\ddots&\vdots&\vdots\\
0& 0&\cdots& 1& p_{\ell-1}
\end{pmatrix}.\]

The constants $p_i$ depend on $\alpha$, and are called the \emph{feedback coefficients} for the LFSR. Note that since $F^*$ is generated by $\alpha$, repeatedly applying this operation to some non-zero initial vector generates every non-zero vector in $\GF(q)^\ell$.

A \emph{Fibonacci LFSR} is a similar construction that uses the transposed state change matrix $M^T$. In this configuration, the next state is given by $\bm{v}\mapsto(v_1,v_2,\dots,v_{\ell-1},\sum_{i=0}^{\ell-1} p_iv_i)$. In fact, a Fibonacci LFSR performs the same operation as the corresponding Galois LFSR when the vectors are represented in a different basis. To see this, we must find a matrix $C$ that satisfies $CM=M^TC$.

Let $C$ be defined by $C_{ij}=(M^i)_{0j}$ for $0\leq i<\ell$ and $0\leq j<\ell$ (that is, the $(i,j)$-th entry of $C$ is the $(0,j)$-th entry of $M^i$). The entries of powers of a companion matrix are explicitly known \cite{chen1996}, so we can observe that $C$ is a symmetric matrix. Similarly, since $(CM)_{ij}=(M^{i+1})_{0j}$, $CM$ is also symmetric. Hence $CM=(CM)^T=M^TC^T=M^TC$, so $C$ is a change of basis matrix from $M$ to $M^T$. Note that the first row of $C$ is $\begin{pmatrix}1& 0 &\cdots& 0\end{pmatrix}$, so the first basis vector for the Fibonacci LFSR is also $\alpha^0=1$ (as in the Galois LFSR).

We now show that the Fibonacci LFSR follows edges in the corresponding de~Bruijn graph. From here onwards, we do not use Galois LFSRs and instead only represent field elements in the Fibonacci basis.

\begin{proposition}\label{prop:edges in fibonacci basis}
	Let $F:=\GF(q^\ell)$ and fix a primitive $\alpha\in F^*$. If the elements of $F$ are identified with the vertices of $\dB(q,\ell)$ by expressing them in the Fibonacci basis over $GF(q)$, then $(\beta,\alpha\beta)$ is an edge in $\dB(q,\ell)$ for each $\beta\in F$.
\end{proposition}
\begin{proof}
	Observe that the state change operation $\beta\mapsto\alpha\beta$ in a Fibonacci LFSR corresponds to a left shift of the state vector and an extra term on the end. This is exactly what is required for an edge in $\dB(q,\ell)$.
\end{proof}

Proposition~\ref{prop:edges in fibonacci basis} can be used to describe all edges of $\dB(q,\ell)$ in terms of field operations.

\begin{lemma}\label{lem:adding scalar preserves edges}
	If $\beta\in F$ and $e\in\GF(q)$, then $(\beta+e,\alpha\beta)$ is an edge in $\dB(q,\ell)$.
\end{lemma}
\begin{proof}
	By Proposition~\ref{prop:edges in fibonacci basis}, $(\beta,\alpha\beta)$ is an edge in $\dB(q,\ell)$. Recall that in the construction of the Fibonacci basis, the first (or leftmost) component corresponds to the basis vector $1$ (the multiplicative identity of $F$). Thus adding a scalar $e$ to a vector only changes the first component, which is shifted out when following an edge in $\dB(q,\ell)$. Hence $\beta$ and $\beta+e$ have the same out-neighbours (including $\alpha\beta$).
\end{proof}

Now consider repeatedly applying the state change operation $\beta\mapsto\alpha\beta$ to some initial non-zero field element ($1$ for example). Since $\alpha$ generates $F^*$, the Fibonacci LFSR traverses a cycle of length $q^\ell-1$ in $\dB(q,\ell)$. The missing vertex is $0$, and can always be inserted into this cycle by replacing the edge $(1,\alpha)$ with two edges $(1,0)$ and $(0,\alpha)$. Note that $(1,0)$ and $(0,\alpha)$ are edges by Lemma~\ref{lem:adding scalar preserves edges} (with $\beta=1$, $e=-1$ and $\beta=0$, $e=1$, respectively). Thus we have found a Hamiltonian cycle in $\dB(q,\ell)$, which is a de~Bruijn sequence.

\subsection{Splitting LFSR sequences}\label{sec:lfsr splitting}
Due to the inherently algebraic construction of linear feedback shift registers, the symmetry properties of such sequences may be exploited to produce many cycles of the same length. For this section, we identify vertices of $\dB(q,\ell)$ with the elements of $F=\GF(q^\ell)$ via the Fibonacci basis described above (with respect to a fixed primitive $\alpha\in F^*$).

Fix a value of $k<q^\ell$, and let $\beta_e:=\frac{\alpha e}{\alpha^k-1}$ for each non-zero scalar $e\in\GF(q)^*$. Since $\alpha^{-1}\beta_e+e=\alpha^{k-1}\beta_e$, there is an edge from $\alpha^{k-1}\beta_e$ to $\beta_e$ in $\dB(q,\ell)$ (by Lemma~\ref{lem:adding scalar preserves edges} with $\beta=\alpha^{-1}\beta_e$). We also have $k-1$ other edges $(\alpha^i\beta_e,\alpha^{i+1}\beta_e)$ for $0\leq i<k-1$, so we can form a $k$-cycle $C_e=(\beta_e,\alpha\beta_e,\dots,\alpha^{k-1}\beta_e)$ for each of the $q-1$ values of $e\in\GF(q)^*$. In general, these $q-1$ cycles are not necessarily pairwise disjoint, but we now show that they are if $k$ is small.

\begin{theorem}\label{thm:lfsr splits evenly}
$\mathcal{E}(q,k,\ell)\geq q-1$ for every prime power $q$ and every $k\leq m:=\frac{q^\ell-1}{q-1}$.
\end{theorem}
\begin{proof}
Let $\log_\alpha\beta\in\mathbb{Z}_{q^\ell-1}$ be the value of $i$ for which $\alpha^i=\beta$; that is, the discrete logarithm of $\beta$ with base $\alpha$. This is well-defined on all of $F^*$ because $\alpha$ is a generator. Note that $\log_\alpha\beta$ is also the position of $\beta$ in the LFSR sequence (if the initial state is 1). Now consider the relative position of the starting points of two different cycles $C_x$ and $C_y$ as described above, with $x,y\in\GF(q)^*$. The distance along the LFSR sequence between these starting points is

\[\log_\alpha\beta_x-\log_\alpha\beta_y=\log_\alpha\left(\frac{\alpha x}{\alpha^k-1}\right)-\log_\alpha\left(\frac{\alpha y}{\alpha^k-1}\right)=\log_\alpha\left(\frac{x}{y}\right).\]

Note that $\frac{x}{y}\in\GF(q)^*$, which is a subgroup of $F^*$ of order $q-1$. Also note that $(\alpha^m)^{q-1}=1$, and that $im<q^\ell-1=|F^*|$ for $i<q-1$. Thus $\alpha^m$ has order $q-1$. But there is only one subgroup of $F^*$ of order $q-1$ (since $F^*$ is cyclic), so $\frac{x}{y}\in\langle\alpha^m\rangle$. Hence $\frac{x}{y}=(\alpha^m)^j$ for some $j$, so $\log_\alpha(\frac{x}{y})=jm$ is an integer multiple of $m$. Since $k\leq m$, the $k$ consecutive vertices of $C_x$ cannot be in $C_y$, whose starting vertex is at least $m$ places past the start of $C_x$. Hence these $q-1$ $k$-cycles are pairwise disjoint.
\end{proof}

\subsection{Translating LFSRs}\label{sec:lfsr translations}
Fix a scalar $e\in\GF(q)$, and let $\xi_e(\beta):=\alpha\beta+\alpha e$. Note that $\xi_e$ has exactly one fixed point, namely $\varphi_e=\frac{\alpha e}{1-\alpha}$.  Hence
\begin{equation}\label{eq:lfsr edge translation}
\xi_e(\beta+\varphi_e)=\alpha\,(\beta+\varphi_e)+\alpha e=\alpha\beta+\xi_e(\varphi_e)=\xi_0(\beta)+\varphi_e.
\end{equation}

By Lemma~\ref{lem:adding scalar preserves edges}, the $q$ out-neighbours of a vertex $\beta$ are $\{\xi_e(\beta)\mid e\in\GF(q)\}$. Thus we may partition the edges of $\dB(q,\ell)$ into the $q$ parts $P_e:=\{(\beta,\xi_e(\beta))\mid\beta\in\GF(q^\ell)\}$, where $e\in\GF(q)$. Note that \eqref{eq:lfsr edge translation} ensures that $(x,y)\in P_0$ implies $(x+\varphi_e,y+\varphi_e)\in P_e$. Hence if $(\beta_1,\beta_2,\dots,\beta_k,\beta_1)$ is a circuit contained in $P_0$, then $(\beta_1+\varphi_e,\beta_2+\varphi_e,\dots,\beta_k+\varphi_e,\beta_1+\varphi_e)$ is a circuit contained in $P_e$. We call this operation \emph{translating} the circuit by $e$. Note that the $q$ translations of a circuit contained in $P_0$ are pairwise edge-disjoint because the $P_e$ are disjoint.

We are now ready to prove Theorem~\ref{thm:lfsr translation}, which we restate here:

\lfsrtranslation*
\begin{proof}
If $\ell=1$, then the result is trivial since $\dB(q,1)$ contains $q$ loops. Hence we may assume that $\ell\geq 2$.

Let $k=q^{\ell-1}$. By Proposition~\ref{prop:partition conditions}, it is sufficient to find a partition of $\dB(q,\ell-1)$ into $q$ edge-disjoint $k$-circuits.

Recall from Section~\ref{sec:dB sequence from LFSR} that the LFSR sequence is constructed using edges solely of the form $(\beta,\alpha\beta)\in P_0$, and forms a cycle $C_0$ of length $k-1$. This cycle (of vertices) is also a circuit of $k-1$ edges, and we can construct a translated $(k-1)$-circuit $C_e$ in $P_e$ for each scalar $e\in\GF(q)$.

The circuit $C_0$ contains every edge from $P_0$ except the loop $(0,0)=(\varphi_0,\varphi_0)$, which translates to another loop $(\varphi_e,\varphi_e)$ in $P_e$. Since $C_0$ contains every non-zero vertex, we may insert, say, the loop $(\varphi_1,\varphi_1)$ into $C_0$ to obtain a $k$-circuit $\widehat{C}_0$ (note that $\widehat{C}_0$ is no longer a cycle since it contains the vertex $\varphi_1$ twice). Similarly, we may insert the loop $(\varphi_{e+1},\varphi_{e+1})$ into $C_e$ to generate a $k$-circuit $\widehat{C}_e$ for each $e\in\GF(q)$.

Observe that each edge $(\beta,\xi_e(\beta))$ of $\dB(q,\ell-1)$ appears in a unique circuit: if $\beta=\varphi_e$, then the edge is a loop in $\widehat{C}_{e-1}$; otherwise it is in $\widehat{C}_e$. Thus we have a partition of $\dB(q,\ell-1)$ into $q$ edge-disjoint $k$-circuits $\{\widehat{C}_e\mid e\in\GF(q)\}$. Hence $\mathcal{E}(q,q^{\ell-1},\ell)\geq q$.
\end{proof}

For example, we may apply Theorem~\ref{thm:lfsr translation} with $q=3$ and $\ell=4$ to obtain three cycles of length 9. Suppose we choose a primitive $\alpha\in F=\GF(3^3)$ whose minimal polynomial over $\GF(3)$ is $\alpha^3=2+\alpha$. We now construct a 26-cycle $C_0$ in $\dB(3,3)$ by iterating the LFSR starting from vertex $100$, which corresponds to the field element $1$ (the multiplicative identity of $F$). This produces the following cycle of symbols: $C_0=10020212210222001012112011$ (the corresponding cycle of vertices is $100,001,002,\dots,011,111,110$).

Now we may construct the 27-circuit $\widehat{C}_0$ by inserting a loop, say $(111,111)$, into $C_0$. The three translations of $\widehat{C}_0$, shown below in \eqref{eq:translating example}, partition the edge set of $\dB(3,3)$, and hence the corresponding cycles in $\dB(3,4)$ partition the vertex set of $\dB(3,4)$. Thus we have shown that $\mathcal{E}(3,27,4)=3$.

\begin{align}
\widehat{C}_0&=100202122102220010121120111,\nonumber\\
\widehat{C}_1&=211010200210001121202201222,\nonumber\\
\widehat{C}_2&=022121011021112202010012000.\label{eq:translating example}
\end{align}

\subsection{LFSRs from non-primitive elements}\label{sec:non-primitive lfsr}
Suppose that in the construction of the LFSR, we chose a non-primitive element $\beta$ with multiplicative order $k<q^\ell-1$. If $\{1,\beta,\beta^2,\dots,\beta^{\ell-1}\}$ is still a basis of $F=\GF(q^\ell)$ over $\GF(q)$, then vectors with respect to this basis are still in correspondence with field elements. Repeated multiplication by $\beta$, however, no longer generates every element of $F^*$; instead this process traverses the cyclic subgroup of order $k$ generated by $\beta$. Thus the action of the LFSR traces out this subgroup of $F^*$ if the initial state is the identity $1$. This corresponds to a $k$-cycle in $\dB(q,\ell)$.

Choosing a different starting state for the LFSR translates the whole sequence, but does not change the length of the cycle. This gives a partition of the non-zero vertices into $k$-cycles. The number of these cycles is $|F^*/\langle\beta\rangle|=\frac{q^\ell-1}{k}$, giving $\mathcal{E}(q,k,\ell)\geq\frac{q^\ell-1}{k}$.

\begin{theorem}\label{thm:non-primitive lfsrs}
Let $q$ be a prime power, and $k$ a factor of $q^\ell-1$. If $k$ does not divide $q^i-1$ for each $i<\ell$, then
\[\mathcal{E}(q,k,\ell)=\frac{q^\ell-1}{k}.\]
\end{theorem}
\begin{proof}
Since $k$ divides $q^\ell-1$, there is an element $\beta$ of multiplicative order $k$ in $F=\GF(q^\ell)$. Since $k$ does not divide $q^i-1$ for $i<\ell$, $\beta$ is not in any subfield $\GF(q^i)$ of $F$. Hence $\beta$ is not the root of any polynomial over $\GF(q)$ of degree $d<\ell$. Therefore $\{1,\beta,\beta^2,\dots,\beta^{\ell-1}\}$ is linearly independent over $\GF(q)$, and hence a basis of $F$.

Therefore, the LFSR generated by $\beta$ traces out a distinct $k$-cycle in $\dB(q,\ell)$ for each equivalence class of $F^*/\langle\beta\rangle$. We have found $\frac{q^\ell-1}{k}$ disjoint $k$-cycles in $\dB(q,\ell)$. Since the $k$-cycles cover all but one vertex in $\dB(q,\ell)$ and $k>1$, this is the best possible bound.
\end{proof}

By Zsigmondy's theorem \cite{zsigmony1892}, there is a prime $p$ that divides $q^\ell-1$ but not $q^i-1$ for $i<\ell$ for any $q$ and $\ell$, except when $(q,\ell)=(2,6)$ or $\ell=2$ and $q$ is a Mersenne prime (that is, $q=2^{p'}-1$ for some prime $p'$). Thus Theorem~\ref{thm:non-primitive lfsrs} can be applied with $k=p$ to obtain an almost optimal eBug colouring with $p$ LEDs on each eBug (only one colour sequence is unused). Furthermore, if larger eBugs are desired for the same values of $q$ and $\ell$, any multiple of $p$ that divides $q^\ell-1$ can also be used for $k$.

\section{Necklaces}\label{sec:necklaces}
This section focuses on the combinatorics of words to find and combine cycles in a de~Bruijn graph. A \emph{word} of length $k$ over an alphabet $A$ is a sequence of $k$ \emph{letters}, each of which is an element of $A$. We often use the \emph{left rotation} operation $\rho$, which cyclically permutes the order of letters in a word: $\rho(a_1a_2\dots a_k):=a_2a_3\dots a_k a_1$. We may rotate a word by any amount by repeatedly applying $\rho$; $\rho^i$ rotates a word by $i$ places to the left.

Usually, a \emph{factor} $f$ of a word $w$ is defined as any block of consecutive letters in $w$, and $f$ is a \emph{prefix} if it appears at the start of $w$. In this case, if $w$ has length $k$, there are at most $k-\ell$ factors of length $\ell$ (or $\ell$-factors) of $w$. For this section, we allow factors to ``wrap around'', so that $f$ is a factor of $w$ if and only if it is a prefix of some rotation $\rho^i(w)$. This way, it is possible to have $k$ different $\ell$-factors of a word of length $k$.

\subsection{Necklaces and de~Bruijn graphs}\label{sec:necklaces intro}
\begin{definition}
	A $q$-ary \emph{necklace} is an equivalence class of words over the alphabet $\mathbb{Z}_q$ under cyclic rotation $\rho^i$. The \emph{length} of a necklace is the length of any word in the class, while the \emph{size} of a necklace is the number of words in the class. A necklace with equal length and size is called \emph{aperiodic}.
\end{definition}

Every word in a $q$-ary length $\ell$ necklace corresponds to a vertex in the de~Bruijn graph $\dB(q,\ell)$, and a cyclic rotation corresponds to following an edge in this graph. Thus a size $t$ necklace can be thought of as a $t$-cycle in $\dB(q,\ell)$. Note that every vertex is part of some necklace, so the vertex set of $\dB(q,\ell)$ can be partitioned into necklaces.

For a fixed $q$ and $\ell$, the possible necklace sizes in $\dB(q,\ell)$ are the divisors of $\ell$. Moreau's necklace counting function \cite{moreau1872}, shown below in \eqref{eq:moreau}, gives the number of $q$-ary size $t$ necklaces, and is defined in terms of the M\"obius function $\mu$.
\begin{equation}\label{eq:moreau}
M(q,t)=\frac{1}{t}\sum\limits_{d\mid t}\mu\left(\frac{t}{d}\right)q^{d}.
\end{equation}

The total number of length $\ell$ necklaces is more easily calculated using Euler's totient function, $\varphi$:
\[Z(q,\ell)=\sum\limits_{t\mid\ell}M(q,t)=\frac{1}{\ell}\sum\limits_{d\mid\ell}\varphi\left(\frac{\ell}{d}\right)q^{d}.\]

When $\ell$ is prime, there are exactly $q$ necklaces of size $1$ (the constant words); the remainder of the necklaces have size $\ell$. Thus there are $\frac{q^\ell-q}{\ell}$ disjoint $\ell$-cycles in $\dB(q,\ell)$. Hence $\mathcal{E}(q,\ell,\ell)\geq\frac{q^\ell-q}{\ell}$ for any prime $\ell$. Note that when $\ell>q$, this lower bound is tight since there are less than $\ell$ remaining vertices in $\dB(q,\ell)$.

A \emph{Lyndon word} is the lexicographically smallest representative of an aperiodic necklace. It is possible to construct a de~Bruijn sequence for $\dB(q,\ell)$ by concatenating all $q$-ary Lyndon words whose length divides $\ell$ in lexicographic order. In fact, the sequence that is generated is the lexicographically smallest de~Bruijn sequence of the given order \cite{fredricksen1970}.

\subsection{Multiplying necklaces}\label{sec:product colouring}
Suppose we have two systems of coloured eBugs, where each colouring is $\ell$-valid. In this section, we describe a type of direct product that yields many more eBugs at the expense of using more colours. The result is summarised in Theorem~\ref{thm:product colouring}.

Instead of modelling the colouring problem with de~Bruijn graphs, we find a set of necklaces of length $k$ that correspond to the disjoint $k$-cycles in $\dB(q,\ell)$. The definition of an \emph{$\ell$-valid} set of necklaces translates directly from Definition~\ref{def:ebug number}.

\productcolouring*
\begin{proof}
We first demonstrate this proof for the special case of $k_1=k_2=k$ (so $\gcd(k_1,k_2)=\lcm(k_1,k_2)=k$), and then show that the construction can be extended to the general case. The construction describes a one-to-$k$ mapping from pairs of necklaces to necklaces with $q_1q_2$ colours.

In order to construct necklaces over a larger alphabet, we use pairs of letters (colours) as the letters in the resulting necklaces. We define a \emph{merging} operation $\mathcal{M}$ that pairs corresponding letters from two words of the same length: if $a:=a_1a_2\dots a_k$ and $b:=b_1b_2\dots b_k$, then $\mathcal{M}(a,b)=(a_1,b_1)(a_2,b_2)\dots(a_k,b_k)$.

Let $N_i$ be an $\ell$-valid set of $\mathcal{E}_i$ $q_i$-ary necklaces of length $k$, for $i=1,2$. For each necklace $n\in N_1\cup N_2$, choose a representative word $w_n$. Now, for a pair of necklaces $(n_1,n_2)\in N_1\times N_2$, we construct $k$ new words $\mathcal{M}(w_{n_1},\rho^i(w_{n_2}))$ over $\mathbb{Z}_{q_1}\times\mathbb{Z}_{q_2}$, where $i\in\mathbb{Z}_k$. Note that we only rotate one of the words, since rotating both by the same amount creates a word that is equivalent under cyclic rotation (it is only the relative rotation that matters). An example of this process with $k=8$, $\ell=3$ and $q_1=q_2=2$ is illustrated in Figure~\ref{fig:product colouring}.

\begin{figure}[htb]
\centering
\adjustbox{valign=M}{
\begin{tikzpicture}[scale=\linewidth/18cm]
\colorlet{a}{yellow!100!white}
\colorlet{b}{blue!50!black}
\draw (0,0) circle [radius=2.25cm];
\foreach\colour[count=\i,evaluate=\colour using \colour*100] in {0,0,0,1,0,1,1,1}
	\node[draw,circle,shading=ball,ball color=a!\colour!b,radius=0.1cm] at ({360/8*\i}:2.25cm) {};
\draw (0,0) circle [radius=1.5cm];
\foreach\colour[count=\i,evaluate=\colour using \colour*100] in {0,1,0,1,1,1,0,0}
	\node[draw,circle,shading=ball,ball color=a!\colour!b,radius=0.1cm] at ({360/8*\i}:1.5cm) {};
\end{tikzpicture}
}
\adjustbox{valign=M}{
\begin{tikzpicture}[scale=\linewidth/18cm]
\colorlet{a}{yellow!100!white}
\colorlet{b}{blue!50!black}
\def\outerseq{0,0,0,1,0,1,1,1}
\foreach\x/\y/\seq in {
	0/0/{0,0,0,1,0,1,1,1},
	1/0/{0,0,1,0,1,1,1,0},
	2/0/{0,1,0,1,1,1,0,0},
	3/0/{1,0,1,1,1,0,0,0},
	0/1/{0,1,1,1,0,0,0,1},
	1/1/{1,1,1,0,0,0,1,0},
	2/1/{1,1,0,0,0,1,0,1},
	3/1/{1,0,0,0,1,0,1,1}}
{
	\begin{scope}[xshift=3cm*\x,yshift=-3cm*\y]
	\draw (0,0) circle [radius=1cm];
	\foreach\colour[count=\i,evaluate=\colour using \colour*100] in \seq
	{
		\node[draw,circle,inner sep=0.1cm] at (360/8*\i:1cm) {};
		\begin{scope}
		\clip ($(360/8*\i:1cm)+(0.5cm,-0.25cm)$) rectangle ($(360/8*\i:1cm)+(0,0.25cm)$);
		\node[circle,shading=ball,ball color=a!\colour!b,inner sep=0.1cm] at (360/8*\i:1cm) {};
		\end{scope}
	}
	\foreach\colour[count=\i,evaluate=\colour using \colour*100] in \outerseq
	{
		\begin{scope}
		\clip ($(360/8*\i:1cm)+(-0.5cm,-0.25cm)$) rectangle ($(360/8*\i:1cm)+(0,0.25cm)$);
		\node[circle,shading=ball,ball color=a!\colour!b,inner sep=0.1cm] at (360/8*\i:1cm) {};
		\end{scope}
	}
	\end{scope}
}
\end{tikzpicture}
}
\caption{Two necklaces being multiplied to produce many new necklaces. Ordered pairs of colours are used to specify the new colours in the resulting necklaces.}\label{fig:product colouring}
\end{figure}
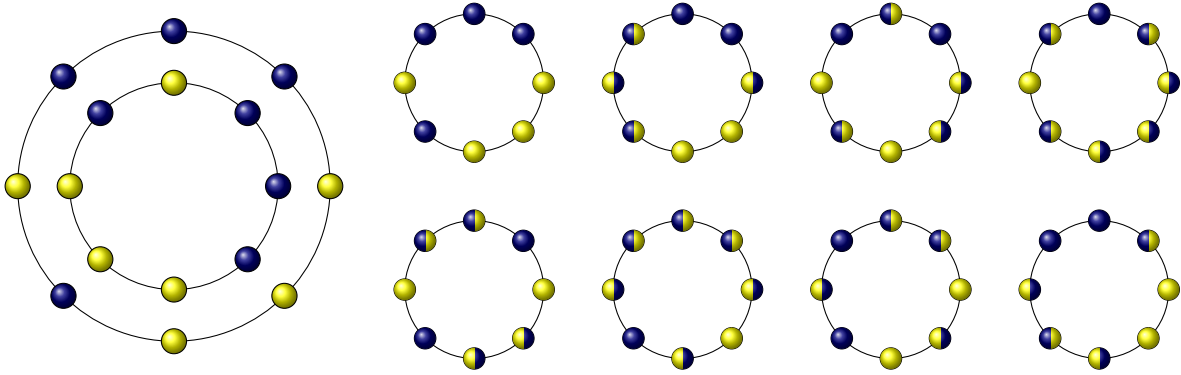

This process can be performed for every pair $(n_1,n_2)\in N_1\times N_2$, generating $k$ new words every time. Hence it is clear that $k\mathcal{E}_1\mathcal{E}_2$ words are produced, and that $q_1q_2$ colours are used. Thus it remains only to show that the set of corresponding necklaces is $\ell$-valid.

In each of the original necklaces $n\in N_1\cup N_2$, there are $k$ different $\ell$-factors (since the $N_i$ were $\ell$-valid). Thus the total number of distinct $\ell$-factors in $N_i$ is $k\mathcal{E}_i$. Suppose the word $a:=a_1a_2\dots a_\ell$ occurs in the necklace $n_a\in N_1$, and $b:=b_1b_2\dots b_\ell$ occurs in $n_b\in N_2$. There is a unique $i$ such that $w_{n_a}$ and $\rho^i(w_{n_b})$ have $a$ and $b$ aligned, so $\mathcal{M}(w_{n_a},\rho^i(w_{n_b}))$ must contain the factor $(a_1,b_1)(a_2,b_2)\dots(a_\ell,b_\ell)$. Since there are $k\mathcal{E}_1\times k\mathcal{E}_2$ pairs of $\ell$-factors from the original necklaces, there must be at least $k\mathcal{E}_1\times k\mathcal{E}_2$ distinct $\ell$-factors in the set of merged words. But there are only $k$ possible $\ell$-factors in each of the $k\mathcal{E}_1\mathcal{E}_2$ merged words, so each $\ell$-factor must appear exactly once. Therefore the set of necklaces corresponding to the merged words is $\ell$-valid.

To generalise to the case where $k_1\neq k_2$, we can traverse the original necklaces multiple times to obtain words of length $\lcm(k_1,k_2)$. For a necklace $n\in N_i$, pick a representative word of $n$ and repeat it $\frac{\lcm(k_1,k_2)}{k_i}$ times to obtain $w_n$. This way we may still merge words from $N_1$ and $N_2$ using $\mathcal{M}$ (since they are the same length).

To model rotations of $w_{n_1}$ and $w_{n_2}$, we act on the pair with elements of the group $\mathbb{Z}_{k_1}\times\mathbb{Z}_{k_2}$. Simultaneous rotation of $w_{n_1}$ and $w_{n_2}$ (by the same amount) only rotates the merged word that is produced, so we may identify unique merged words with elements of the quotient group $R:=\mathbb{Z}_{k_1}\times\mathbb{Z}_{k_2}/\langle(1,1)\rangle$. Hence each pair $(w_{n_1},w_{n_2})$ produces $|R|=\frac{k_1k_2}{\lcm(k_1,k_2)}=\gcd(k_1,k_2)$ unique words, for a total of $\gcd(k_1,k_2)\,\mathcal{E}_1\mathcal{E}_2$ words.

It remains to show that the set of merged words is $\ell$-valid. The number of $\ell$-factors that appear in $N_i$ is $k_i\mathcal{E}_i$. As before, suppose the $\ell$-factor $a$ occurs in $n_a\in N_1$ and $b$ occurs in $n_b\in N_2$. Define $i$ and $j$ so that $a$ is the $\ell$-prefix of $\rho^i(w_{n_a})$, and $b$ is the $\ell$-prefix of $\rho^j(w_{n_b})$. Note that rotating both words together keeps $a$ and $b$ aligned, so the pair $(i,j)$ corresponds to a unique element of $R$, and thus $\mathcal{M}(a,b)$ appears in one of the merged words. Hence there are at least $k_1\mathcal{E}_1\times k_2\mathcal{E}_2$ distinct $\ell$-factors in the set of merged words. The number of possible positions for these $\ell$-factors is the product of the number of merged words and the length of each word. This number is $\gcd(k_1,k_2)\,\mathcal{E}_1\mathcal{E}_2\times\lcm(k_1,k_2)=k_1k_2\mathcal{E}_1\mathcal{E}_2$, so all $\ell$-factors are unique, and the set of corresponding necklaces is $\ell$-valid. Hence $\mathcal{E}(q_1q_2,\lcm(k_1,k_2),\ell)\geq\gcd(k_1,k_2)\,\mathcal{E}_1\,\mathcal{E}_2$.
\end{proof}

The conditions in Theorem~\ref{thm:product colouring} guarantee that if the original colourings are optimal, then the resulting colouring is also optimal. This allows a result for prime powers, such as Theorem~\ref{thm:lfsr translation}, to be extended to any integer by repeated application of Theorem~\ref{thm:product colouring} (after applying Theorem~\ref{thm:lfsr translation} for each prime power factor in the prime decomposition of $q$).

\begin{corollary}\label{cor:dB(q,l) always decomposes into q cycles}
$\mathcal{E}(q,q^{\ell-1},\ell)=q$ for all $q$ and $\ell$.
\end{corollary}

On the other hand, if the original colourings are not optimal, the product colouring may be even ``less'' optimal. For example, the construction in Section~\ref{sec:non-primitive lfsr} is almost optimal since it uses all but one vertex of $\dB(q_1,\ell)$. If we use Theorem~\ref{thm:product colouring} to multiply this with an optimal colouring (where the $k$-cycles use every vertex of $\dB(q_2,\ell)$), there would be $q_2^\ell$ vertices of $\dB(q_1q_2,\ell)$ not used by a cycle in the product colouring.

\subsection{Interleaving necklaces}\label{sec:interleaving}
\interleaving*
\begin{proof}
	Let $N$ be an $\ell$-valid set of $\mathcal{E}(q,k,\ell)$ necklaces of length $k$ over the alphabet $\mathbb{Z}_q$. For each necklace $n\in N$, fix a specific representative word $w_n$. We now construct the set $W\subseteq \mathbb{Z}_q^k$ of all words that appear in some necklace from $N$:
	\[W:=\{\rho^i(w_n)\mid n\in N,i\in\mathbb{Z}_k\}.\]
	
	Since $N$ is $\ell$-valid, the necklaces of $N$ are aperiodic. Thus we may uniquely identify how far each word in $W$ is rotated from its representative $w_n$. For each word $w=\rho^i(w_n)\in W$, define $\psi(w):=i\in\mathbb{Z}_k$.
	
	We now define a function $\mathcal{I}$ that interleaves the letters of multiple words to construct a single long word.
	\[\mathcal{I}(a_{11}a_{12}\dots a_{1k},a_{21}a_{22}\dots a_{2k},\dots,a_{t1}a_{t2}\dots a_{tk}):=a_{11}a_{21}\dots a_{t1}a_{12}a_{22}\dots a_{t2}a_{13}\dots a_{tk}.\]
	
	Note that $\mathcal{I}$ is injective, since we can deinterleave the resulting word to recover the original words. We write $\mathcal{I}^{-1}$ for this deinterleaving function, which gives a $t$-tuple of $k$ words from a single word of length $tk$. For brevity, we write $t$-tuples of words in boldface as $\bm{w}:=(w_1,w_2,\dots,w_t)\in W^t$. We also extend $\psi$ to operate on $t$-tuples, so we may write $\psi(\bm{w}):=(\psi(w_1),\psi(w_2),\dots,\psi(w_t))$.
	
	Now consider the set $V:=\{\mathcal{I}(\bm{w})\mid\bm{w}\in W^t,\psi(w_1)=0,\sum_i \psi(w_i)\equiv 0\mod{t}\}$. Note that the $\psi(w_i)$ are in $\mathbb{Z}_k$ and $t$ divides $k$, so taking the sum modulo $t$ is well-defined. Simple arithmetic shows that $|V|=\frac{|W|^t}{kt}=\frac{k^{t-1}}{t}|N|^t$. We claim that $V$ is a $t\ell$-valid set of words.
	
	Take any $v=\mathcal{I}(\bm{w})$ from $V$, and observe the following property of the interleaving function:
	\begin{equation}\label{eq:rotate interleaved word}
		\rho(v)=\mathcal{I}(w_2,w_3,\dots,w_t,\rho(w_1)).
	\end{equation}
	
	Thus if $\psi(\bm{w})=(z_1,z_2,\dots,z_t)$, then $\psi(\mathcal{I}^{-1}(\rho(v)))=(z_2,z_3,\dots,z_t,z_1+1)$. Hence rotating $v$ to the left increases $\sum\psi(w_i)$ by 1. By iterating \eqref{eq:rotate interleaved word}, we can find all $tk$ rotations of $v$. Of these $tk$ rotations, only $k$ have $\sum \psi(w_i)\equiv 0\mod{t}$ (every $t$-th rotation), and only one of these $k$ has the first word not rotated ($\psi(w_1)=0$). Hence no other $v'\in V$ is a rotation of $v$.
	
	Now suppose there are $v,v'\in V$ that share a $t\ell$-factor $u$ (in any position). Rotate $v$ and $v'$, respectively, to find $\rho^z(v)$ and $\rho^{z'}(v')$, both of which have $u$ as their $t\ell$-prefix. Note that $\rho^z(v)$ and $\rho^{z'}(v')$ are not necessarily in $V$, but are still obtained by interleaving a $t$-tuple of words from $W$ (see \eqref{eq:rotate interleaved word}). We may deinterleave $u$ into a $t$-tuple of length $\ell$ words $\mathcal{I}^{-1}(u)=\bm{u}:=(u_1,u_2,\dots,u_t)$. Each of these words $u_i$ is the $\ell$-prefix of a unique word $w_i\in W$ due to the $\ell$-validity of $N$. Thus the only interleaved word with $u$ as its $t\ell$-prefix is $\mathcal{I}(\bm{w})$, and hence $\rho^z(v)=\mathcal{I}(\bm{w})=\rho^{z'}(v')$. But $v'$ is not a rotation of $v$, so $v=v'$. Therefore, $V$ is a $t\ell$-valid set of length $tk$ words and $\mathcal{E}(q,tk,t\ell)\geq|V|=\frac{k^{t-1}}{t}\mathcal{E}(q,k,\ell)^t$.
\end{proof}

As an example, we now apply Theorem~\ref{thm:interleaving} with $q=4$, $k=8$, $\ell=3$ and $t=2$ to the eight necklaces of length 8 in Figure~\ref{fig:product colouring}, which in turn were constructed using Theorem~\ref{thm:product colouring}. These necklaces may be more compactly written as strings: $00030333,10021233,11020323,11120232,01130223,10131222,01031322,00121332$. Suppose we wish to interleave the first and second necklaces in all allowable rotations. We must keep the first necklace fixed and only rotate the second necklace by even amounts to satisfy the conditions of the set $V$:
\[\begin{array}{cccc}
	00030333& 00030333& 00030333& 00030333\\
	10021233& 02123310& 12331002& 33100212
\end{array}\]

We may now interleave these necklaces to obtain the following four necklaces of length 16:
\[\begin{array}{cccccccc}
	0100003201323333& 0002013203333130& 0102033301303032& 0303013000323132
\end{array}\]

But this is for just one pair of necklaces; we can repeat this procedure for every ordered pair of necklaces from our original list of eight. For each of the 64 possible pairs, we produce four new necklaces of length 16, yielding a total of 256 necklaces. These are listed in Table~\ref{tab:interleaving example}, with each line corresponding to a particular pair of necklaces. Hence $\mathcal{E}(4,16,6)=256$.

\begin{table}
\[\footnotesize
\begin{array}{rccccccccc}
	(0,0):& 0000003300333333& 0003003303333030& 0003033300303033& 0303003000333033 \\
	(0,1):& 0100003201323333& 0002013203333130& 0102033301303032& 0303013000323132 \\
	(0,2):& 0101003200333233& 0002003302333131& 0003023301313032& 0203013100323033 \\
	(0,3):& 0101013200323332& 0102003203323131& 0002033201313132& 0302013101323032 \\
	(0,4):& 0001013300323233& 0103003202333031& 0002023300313133& 0203003101333032 \\
	(0,5):& 0100013301323232& 0103013202323130& 0102023201303133& 0202013001333132 \\
	(0,6):& 0001003301333232& 0003013302323031& 0103023200313033& 0202003100333133 \\
	(0,7):& 0000013201333332& 0102013303323030& 0103033200303132& 0302003001323133 \\
	(1,0):& 1000002310233333& 1003002313233030& 1003032310203033& 1303002010233033 \\
	(1,1):& 1100002211223333& 1002012213233130& 1102032311203032& 1303012010223132 \\
	(1,2):& 1101002210233233& 1002002312233131& 1003022311213032& 1203012110223033 \\
	(1,3):& 1101012210223332& 1102002213223131& 1002032211213132& 1302012111223032 \\
	(1,4):& 1001012310223233& 1103002212233031& 1002022310213133& 1203002111233032 \\
	(1,5):& 1100012311223232& 1103012212223130& 1102022211203133& 1202012011233132 \\
	(1,6):& 1001002311233232& 1003012312223031& 1103022210213033& 1202002110233133 \\
	(1,7):& 1000012211233332& 1102012313223030& 1103032210203132& 1302002011223133 \\
	(2,0):& 1010002300332333& 1013002303332030& 1013032300302033& 1313002000332033 \\
	(2,1):& 1110002201322333& 1012012203332130& 1112032301302032& 1313012000322132 \\
	(2,2):& 1111002200332233& 1012002302332131& 1013022301312032& 1213012100322033 \\
	(2,3):& 1111012200322332& 1112002203322131& 1012032201312132& 1312012101322032 \\
	(2,4):& 1011012300322233& 1113002202332031& 1012022300312133& 1213002101332032 \\
	(2,5):& 1110012301322232& 1113012202322130& 1112022201302133& 1212012001332132 \\
	(2,6):& 1011002301332232& 1013012302322031& 1113022200312033& 1212002100332133 \\
	(2,7):& 1010012201332332& 1112012303322030& 1113032200302132& 1312002001322133 \\
	(3,0):& 1010102300233323& 1013102303233020& 1013132300203023& 1313102000233023 \\
	(3,1):& 1110102201223323& 1012112203233120& 1112132301203022& 1313112000223122 \\
	(3,2):& 1111102200233223& 1012102302233121& 1013122301213022& 1213112100223023 \\
	(3,3):& 1111112200223322& 1112102203223121& 1012132201213122& 1312112101223022 \\
	(3,4):& 1011112300223223& 1113102202233021& 1012122300213123& 1213102101233022 \\
	(3,5):& 1110112301223222& 1113112202223120& 1112122201203123& 1212112001233122 \\
	(3,6):& 1011102301233222& 1013112302223021& 1113122200213023& 1212102100233123 \\
	(3,7):& 1010112201233322& 1112112303223020& 1113132200203122& 1312102001223123 \\
	(4,0):& 0010103300232333& 0013103303232030& 0013133300202033& 0313103000232033 \\
	(4,1):& 0110103201222333& 0012113203232130& 0112133301202032& 0313113000222132 \\
	(4,2):& 0111103200232233& 0012103302232131& 0013123301212032& 0213113100222033 \\
	(4,3):& 0111113200222332& 0112103203222131& 0012133201212132& 0312113101222032 \\
	(4,4):& 0011113300222233& 0113103202232031& 0012123300212133& 0213103101232032 \\
	(4,5):& 0110113301222232& 0113113202222130& 0112123201202133& 0212113001232132 \\
	(4,6):& 0011103301232232& 0013113302222031& 0113123200212033& 0212103100232133 \\
	(4,7):& 0010113201232332& 0112113303222030& 0113133200202132& 0312103001222133 \\
	(5,0):& 1000103310232323& 1003103313232020& 1003133310202023& 1303103010232023 \\
	(5,1):& 1100103211222323& 1002113213232120& 1102133311202022& 1303113010222122 \\
	(5,2):& 1101103210232223& 1002103312232121& 1003123311212022& 1203113110222023 \\
	(5,3):& 1101113210222322& 1102103213222121& 1002133211212122& 1302113111222022 \\
	(5,4):& 1001113310222223& 1103103212232021& 1002123310212123& 1203103111232022 \\
	(5,5):& 1100113311222222& 1103113212222120& 1102123211202123& 1202113011232122 \\
	(5,6):& 1001103311232222& 1003113312222021& 1103123210212023& 1202103110232123 \\
	(5,7):& 1000113211232322& 1102113313222020& 1103133210202122& 1302103011222123 \\
	(6,0):& 0010003310332323& 0013003313332020& 0013033310302023& 0313003010332023 \\
	(6,1):& 0110003211322323& 0012013213332120& 0112033311302022& 0313013010322122 \\
	(6,2):& 0111003210332223& 0012003312332121& 0013023311312022& 0213013110322023 \\
	(6,3):& 0111013210322322& 0112003213322121& 0012033211312122& 0312013111322022 \\
	(6,4):& 0011013310322223& 0113003212332021& 0012023310312123& 0213003111332022 \\
	(6,5):& 0110013311322222& 0113013212322120& 0112023211302123& 0212013011332122 \\
	(6,6):& 0011003311332222& 0013013312322021& 0113023210312023& 0212003110332123 \\
	(6,7):& 0010013211332322& 0112013313322020& 0113033210302122& 0312003011322123 \\
	(7,0):& 0000102310333323& 0003102313333020& 0003132310303023& 0303102010333023 \\
	(7,1):& 0100102211323323& 0002112213333120& 0102132311303022& 0303112010323122 \\
	(7,2):& 0101102210333223& 0002102312333121& 0003122311313022& 0203112110323023 \\
	(7,3):& 0101112210323322& 0102102213323121& 0002132211313122& 0302112111323022 \\
	(7,4):& 0001112310323223& 0103102212333021& 0002122310313123& 0203102111333022 \\
	(7,5):& 0100112311323222& 0103112212323120& 0102122211303123& 0202112011333122 \\
	(7,6):& 0001102311333222& 0003112312323021& 0103122210313023& 0202102110333123 \\
	(7,7):& 0000112211333322& 0102112313323020& 0103132210303122& 0302102011323123 \\
\end{array}\]
\caption{Example of interleaved necklaces. Eight 8-cycles from $\dB(4,3)$ were used to produce these 256 16-cycles in $\dB(4,6)$.}\label{tab:interleaving example}
\end{table}

As with Theorem~\ref{thm:product colouring}, if the original colouring is optimal, then the colouring obtained by interleaving the necklaces is also optimal. This allows us to extend existing results by recursively applying Theorem~\ref{thm:interleaving}.

\begin{corollary}
	If every prime factor of $t$ divides $q$, then
	\begin{align*}
		\mathcal{E}(q,tq^\ell,t\ell)&=\frac{q^{(t-1)\ell}}{t}\text{, and}\\
		\mathcal{E}(q,tq^{\ell-1},t\ell)&=\frac{q^{(t-1)\ell+1}}{t}.	
	\end{align*}
\end{corollary}
\begin{proof}
	Recall that $\mathcal{E}(q,q^\ell,\ell)=1$ and $\mathcal{E}(q,q^{\ell-1},\ell)=q$ (see Corollary~\ref{cor:dB(q,l) always decomposes into q cycles}). Now apply Theorem~\ref{thm:interleaving} repeatedly using each prime factor of $t$. Note that since each prime factor of $t$ divides $q$, and each value of $k$ in this process is a multiple of $q$, Theorem~\ref{thm:interleaving} is applicable at each step.
\end{proof}

We also have a partial extension of Theorem~\ref{thm:interleaving}, which allows us to interleave a pair of necklaces of odd length:
\begin{theorem}\label{thm:interleaving 2}
	$\mathcal{E}(q,2k,2\ell)\geq\left\lfloor\frac{k}{2}\right\rfloor\mathcal{E}(q,k,\ell)^2$.
\end{theorem}
The proof of Theorem~\ref{thm:interleaving 2} is analogous to that of Theorem~\ref{thm:interleaving}, except that a slightly different condition is used in the construction of the set $V$. Here, $V:=\{\mathcal{I}(w_1,w_2)\mid w_1,w_2\in W,\psi(w_1)=0,\psi(w_2)<\frac{k-1}{2}\}$.

\subsection{Necklace concatenation}
Whenever two length $\ell$ necklaces share an $(\ell-1)$-factor, the corresponding cycles in $\dB(q,\ell)$ can be concatenated. This is because the corresponding edge circuits in $\dB(q,\ell-1)$ have a common vertex, and can thus be joined to create a larger circuit, which in turn gives a larger cycle in $\dB(q,\ell)$. This relationship between necklaces turns out to be very useful, so we construct a \emph{necklace adjacency graph} $N(q,\ell)$. The $q$-ary length $\ell$ necklaces form the vertex set of $N(q,\ell)$, while pairs of necklaces that share an $(\ell-1)$-factor are joined by an edge.

Consider any (connected) subtree $S$ in $N(q,\ell)$. By applying the above operation for each edge in $S$, the cycles for each necklace in $S$ can be concatenated together to produce one long cycle, whose length is the sum of the sizes of the individual necklaces. Hence if we find a spanning forest in $N(q,\ell)$ where each component subtree has $k$ as the total size of its necklaces, we can partition $\dB(q,\ell)$ into $k$-cycles. If the forest does not span $N(q,\ell)$, this still gives a lower bound on the eBug number: if there are $m$ component subtrees in the forest, each with total necklace size $k$, then $\mathcal{E}(q,k,\ell)\geq m$.

In the case when $\ell$ is prime, recall that most necklaces have size $\ell$. Let $N'(q,\ell)$ be the subgraph of $N(q,\ell)$ induced by the size $\ell$ necklaces. Suppose that there is a perfect matching in $N'(q,\ell)$: this is a spanning forest, and each component subtree has total size $2\ell$. Hence concatenating the cycles for each pair of necklaces in the matching produces $\frac{q^\ell-q}{2\ell}$ cycles in $\dB(q,\ell)$, each of length $2\ell$. Similarly, we may generalise this to larger multiples of $\ell$:

\begin{proposition}\label{prop:necklace joining}
	Let $\ell$ be a prime. If there is a spanning forest of $N'(q,\ell)$ in which each component subtree has $t$ vertices, then $\mathcal{E}(q,t\ell,\ell)\geq\frac{q^\ell-q}{t\ell}$, with equality if $q<t\ell$.
\end{proposition}

In particular, the existence of a Hamiltonian path in $N'(q,\ell)$ is sufficient to apply Proposition~\ref{prop:necklace joining} for any $t$ that divides $\frac{q^\ell-q}{\ell}$ (removing every $t$-th edge from the path produces the required forest). A Hamiltonian path in $N(q,\ell)$ is called a \emph{Gray code} for necklaces, and is conjectured to exist whenever $q$ or $\ell$ is odd \cite{savage1997} (if both $q$ and $\ell$ are even, a simple parity argument reveals that $N(q,\ell)$ is bipartite with unequal parts). It is a trivial matter to transform the Hamiltonian path in $N(q,\ell)$ to one in $N'(q,\ell)$, since all neighbours of a constant necklace are adjacent to each other. For $q=2$, there is also an existing construction of a \emph{2-Gray code} for fixed density necklaces \cite{wang1996}, which lists every necklace with a fixed number of $0$s such that consecutive necklaces differ in exactly two places (a $0$ and $1$ are exchanged).

\subsection{Robot identification with necklaces}
In some applications of robot identification, orientation information is either not required or can be obtained through different means \cite{wang2014,wang2013a}. This relaxes the conditions necessary to obtain a valid colouring of eBugs, since a particular colour sequence may appear more than once on a single eBug. In the de~Bruijn graph, this amounts to finding the maximum number of disjoint closed $k$-walks (instead of $k$-cycles). When $\ell$ is a divisor of $k$, there is an easy solution that turns out to be the best possible; this is a direct corollary of Golomb's conjecture.

\begin{proposition}\label{prop:necklaces for identification}
If $\ell$ is a divisor of $k$, then the maximum number of pairwise disjoint closed $k$-walks in $\dB(q,\ell)$ is $Z(q,\ell)$, the number of $q$-ary length $\ell$ necklaces.
\end{proposition}
\begin{proof}
Golomb's conjecture, which was proved by Mykkeltveit \cite{mykkeltveit1972}, states that the maximum number of pairwise disjoint cycles (of any length) in $\dB(q,\ell)$ is $Z(q,\ell)$. Since each closed walk contains a cycle, there are at most $Z(q,\ell)$ pairwise disjoint closed $k$-walks in $\dB(q,\ell)$.

Now consider any length $\ell$ necklace. The size $t$ of this necklace must divide $\ell$ and hence $k$, so the corresponding $t$-cycle in $\dB(q,\ell)$ can be traversed multiple times to obtain a closed $k$-walk. Thus there are $Z(q,\ell)$ pairwise disjoint closed $k$-walks in the graph.
\end{proof}

\section{Concluding remarks}
The theorems presented in Sections~\ref{sec:LFSRs} and \ref{sec:necklaces} focus on finding large eBug colourings to obtain bounds on the eBug number $\mathcal{E}(q,k,\ell)$. In particular, we concentrated on constructions that yield optimal colourings to support Conjecture~\ref{conj:E always optimal}. The algebraic construction in Theorem~\ref{thm:lfsr translation}, and its corresponding extension in Corollary~\ref{cor:dB(q,l) always decomposes into q cycles}, produces $q$ eBugs of maximal size for any $q$ and $\ell$. The combinatorial results in Section~\ref{sec:necklaces} increase the number of eBugs by moving to a larger de~Bruijn graph. These results can even produce eBug colourings for practical applications: for example, we have $\mathcal{E}(2,16,5)=2$ by Theorem~\ref{thm:lfsr translation}, so $\mathcal{E}(4,16,5)=64$ by Theorem~\ref{thm:product colouring}. The current eBugs have 16 LEDs, a camera can distinguish four colours in an image quite easily, and five consecutive LEDs are visible in practice, so it is possible to construct a network of 64 uniquely identifiable eBugs. Also, Theorem~\ref{thm:interleaving} can be used to significantly increase the number of eBugs without increasing the number of colours, $q$. Moreover, it keeps the ratio $\frac{k}{\ell}$ constant, which is a reasonable assumption when designing such a network (since the camera can see a fixed arc of the circle of LEDs).

Unfortunately, the only way to produce many eBugs (more than $q$) with our results is by applying Theorem~\ref{thm:product colouring} or Theorem~\ref{thm:interleaving}, which necessarily increase either $q$, the number of colours, or $\ell$, the number of consecutive visible LEDs. The major remaining gap appears to be for prime $q$ and $\ell$, with $k<q^{\ell-1}$, since the multiplying/interleaving constructions cannot produce them.

Ideally, we would like to be able to have many small eBugs (small number of LEDs) with a small number of colours. For example, we would like to show that $\mathcal{E}(q,q^{\ell-i},\ell)=q^i$ when $\ell$ is large enough (for each $i$). Probabilistic arguments may be useful in finding such colourings, but a constructive approach is preferable for applications to robot networks (the search space becomes too large even for small practical examples, and in many cases the colourings appear to be quite rare).

Solving Problem~\ref{prob:determine E} is much harder, since there are no optimal colourings in the cases not covered by Conjecture~\ref{conj:E always optimal}. Improving the lower bounds in these cases, however, is likely to be a much easier task. We have shown that near-optimal colourings exist in many cases.

\section*{Acknowledgements}
We would like to thank Heiko Dietrich for introducing us to Zsigmondy's Theorem, and suggesting that it can be used to find examples to which Theorem~\ref{thm:non-primitive lfsrs} can be applied. We also thank the referees for pointing out some related references and for their helpful suggestions.

\section*{References}

\begin{thebibliography}{10}
	\providecommand{\url}[1]{#1}
	\csname url@samestyle\endcsname
	\providecommand{\newblock}{\relax}
	\providecommand{\bibinfo}[2]{#2}
	\providecommand{\BIBentrySTDinterwordspacing}{\spaceskip=0pt\relax}
	\providecommand{\BIBentryALTinterwordstretchfactor}{4}
	\providecommand{\BIBentryALTinterwordspacing}{\spaceskip=\fontdimen2\font plus
		\BIBentryALTinterwordstretchfactor\fontdimen3\font minus
		\fontdimen4\font\relax}
	\providecommand{\BIBforeignlanguage}[2]{{%
			\expandafter\ifx\csname l@#1\endcsname\relax
			\typeout{** WARNING: IEEEtranS.bst: No hyphenation pattern has been}%
			\typeout{** loaded for the language `#1'. Using the pattern for}%
			\typeout{** the default language instead.}%
			\else
			\language=\csname l@#1\endcsname
			\fi
			#2}}
	\providecommand{\BIBdecl}{\relax}
	\BIBdecl
	
	\bibitem{alspach2003}
	B.~Alspach, H.~Gavlas, M.~{\v{S}}ajna, and H.~Verrall, ``Cycle decompositions.
	{IV}. {C}omplete directed graphs and fixed length directed cycles,'' \emph{J.
		Combin. Theory Ser. A}, vol. 103, no.~1, pp. 165--208, 2003.
	
	\bibitem{anderson2008}
	B.~D.~O. Anderson, C.~Yu, B.~Fidan, and J.~M. Hendrickx, ``Rigid graph control
	architectures for autonomous formations: applying classical graph theory to
	the control of multiagent systems,'' \emph{IEEE Control Systems Magazine},
	vol.~28, no.~6, pp. 48--63, 2008.
	
	\bibitem{bryant2007}
	D.~Bryant, ``Cycle decompositions of complete graphs,'' in \emph{Surveys in
		Combinatorics 2007}, ser. London Math. Soc. Lecture Note Ser.\hskip 1em plus
	0.5em minus 0.4em\relax Cambridge Univ. Press, 2007, vol. 346, pp. 67--97.
	
	\bibitem{bryant2014}
	D.~Bryant, D.~Horsley, and W.~Pettersson, ``Cycle decompositions {V}:
	{C}omplete graphs into cycles of arbitrary lengths,'' \emph{Proc. Lond. Math.
		Soc. (3)}, vol. 108, no.~5, pp. 1153--1192, 2014.
	
	\bibitem{bryant1994}
	D.~E. Bryant, ``Decompositions of directed graphs with loops and related
	algebras,'' \emph{Ars Combin.}, vol.~38, pp. 129--136, 1994.
	
	\bibitem{chen1996}
	W.~Y.~C. Chen and J.~D. Louck, ``The combinatorial power of the companion
	matrix,'' \emph{Linear Algebra Appl.}, vol. 232, pp. 261--278, 1996.
	
	\bibitem{cohn1972}
	M.~Cohn and A.~Lempel, ``Cycle decomposition by disjoint transpositions,''
	\emph{J. Combinatorial Theory Ser. A}, vol.~13, pp. 83--89, 1972.
	
	\bibitem{dademo2011}
	N.~d'Ademo, W.~L.~D. Lui, W.~H. Li, Y.~A. {\c{S}}ekercio{\u{g}}lu, and
	T.~Drummond, ``{eBug} --- an open robotics platform for teaching and
	research,'' in \emph{Australasian Conference on Robotics and
		Automation}.\hskip 1em plus 0.5em minus 0.4em\relax Australian Robotics \&
	Automation Association, 2011.
	
	\bibitem{das2002}
	A.~K. Das, R.~Fierro, V.~Kumar, J.~P. Ostrowski, J.~Spletzer, and C.~J. Taylor,
	``A vision-based formation control framework,'' \emph{IEEE Transactions on
		Robotics and Automation}, vol.~18, no.~5, pp. 813--825, 2002.
	
	\bibitem{debruijn1946}
	N.~G. de~Bruijn, ``A combinatorial problem,'' \emph{Nederl. Akad. Wetensch.,
		Proc.}, vol.~49, pp. 758--764, 1946.
	
	\bibitem{dudeney1917}
	H.~E. Dudeney, \emph{Amusements in mathematics}.\hskip 1em plus 0.5em minus
	0.4em\relax EBD, 1917.
	
	\bibitem{erdos1975}
	P.~Erd{\H{o}}s and L.~Lov\'{a}sz, ``Problems and results on {$3$}-chromatic
	hypergraphs and some related questions,'' in \emph{Infinite and Finite Sets},
	ser. Colloq. Math. Soc. J\'anos Bolyai.\hskip 1em plus 0.5em minus
	0.4em\relax North-Holland, 1975, vol.~10, pp. 609--627.
	
	\bibitem{fraleigh2003}
	J.~Fraleigh and V.~Katz, \emph{A First Course in Abstract Algebra}, ser.
	Addison-Wesley world student series.\hskip 1em plus 0.5em minus 0.4em\relax
	Addison-Wesley, 2003.
	
	\bibitem{fredricksen1970}
	H.~Fredricksen, ``The lexicographically least de {B}ruijn cycle,'' \emph{J.
		Combinatorial Theory}, vol.~9, pp. 1--5, 1970.
	
	\bibitem{good1946}
	I.~J. Good, ``Normal recurring decimals,'' \emph{J. London Math. Soc.},
	vol.~21, pp. 167--169, 1946.
	
	\bibitem{kobayashi1993}
	M.~Kobayashi, Kiyasu-Zen'iti, and G.~Nakamura, ``A solution of {D}udeney's
	round table problem for an even number of people,'' \emph{J. Combin. Theory
		Ser. A}, vol.~63, no.~1, pp. 26--42, 1993.
	
	\bibitem{kobayashi2002}
	M.~Kobayashi and G.~Nakamura, ``Resolvable coverings of 2-paths by cycles,''
	\emph{Graphs Combin.}, vol.~18, no.~4, pp. 739--744, 2002, graph theory and
	discrete geometry (Manila, 2001).
	
	\bibitem{moreau1872}
	C.~Moreau, ``Sur les permutations circulaires distinctes,'' in \emph{Nouvelles
		annales de math{\'e}matiques}, vol.~11.\hskip 1em plus 0.5em minus
	0.4em\relax Gauthier-Villars, 1872, pp. 309--314.
	
	\bibitem{mykkeltveit1972}
	J.~Mykkeltveit, ``A proof of {G}olomb's conjecture for the de {B}ruijn graph,''
	\emph{J. Combinatorial Theory Ser. B}, vol.~13, pp. 40--45, 1972.
	
	\bibitem{nakamura1980}
	G.~Nakamura, Z.~Kiyasu, and N.~Ikeno, ``Solution of the round table problem for
	the case of {$(p^{k}+1)$}\ persons,'' \emph{Comment. Math. Univ. St. Paul.},
	vol.~29, no.~1, pp. 7--20, 1980.
	
	\bibitem{ruskey2003}
	\BIBentryALTinterwordspacing
	F.~Ruskey, ``Combinatorial generation,'' \emph{Working Version (1j-CSC
		425/520)}, 2003. [Online]. Available:
	\url{www.1stworks.com/ref/RuskeyCombGen.pdf}
	\BIBentrySTDinterwordspacing
	
	\bibitem{savage1997}
	C.~Savage, ``A survey of combinatorial {G}ray codes,'' \emph{SIAM Rev.},
	vol.~39, no.~4, pp. 605--629, 1997.
	
	\bibitem{spencer1977}
	J.~Spencer, ``Asymptotic lower bounds for {R}amsey functions,'' \emph{Discrete
		Math.}, vol.~20, no.~1, pp. 69--76, 1977/78.
	
	\bibitem{vanaardenne-ehrenfest1951}
	T.~van Aardenne-Ehrenfest and N.~G. de~Bruijn, ``Circuits and trees in oriented
	linear graphs,'' \emph{Simon Stevin}, vol.~28, pp. 203--217, 1951.
	
	\bibitem{wang1996}
	T.~M.~Y. Wang and C.~D. Savage, ``A {G}ray code for necklaces of fixed
	density,'' \emph{SIAM J. Discrete Math.}, vol.~9, no.~4, pp. 654--673, 1996.
	
	\bibitem{wang2014}
	X.~Wang, Y.~A. \c{S}ekercio\u{g}lu, and T.~Drummond, ``Vision-based cooperative
	pose estimation for localization in multi-robot systems equipped with {RGB-D}
	cameras,'' \emph{Robotics}, vol.~4, no.~1, pp. 1--22, 2015, special issue on
	coordination of robotic systems.
	
	\bibitem{wang2013a}
	X.~Wang, Y.~A. {\c{S}}ekercio{\u{g}}lu, and T.~Drummond, ``A real-time
	distributed relative pose estimation algorithm for {RGB-D} camera equipped
	visual sensor networks,'' in \emph{Proceedings of the Seventh ACM/IEEE
		International Conference on Distributed Smart Cameras (ICDSC 2013)}, Palm
	Springs, California, USA, Oct. 2013.
	
	\bibitem{wsrnlab}
	``Wireless sensor and robot networks laboratory ({WSRNLab}),''
	\url{http://wsrnlab.ecse.monash.edu.au}, {M}onash University, Melbourne,
	Australia.
	
	\bibitem{zhang1987}
	F.~J. Zhang and G.~N. Lin, ``On the de {B}ruijn-{G}ood graphs,'' \emph{Acta
		Math. Sinica}, vol.~30, no.~2, pp. 195--205, 1987.
	
	\bibitem{zsigmony1892}
	K.~Zsigmondy, ``Zur {T}heorie der {P}otenzreste,'' \emph{Monatsh. Math. Phys.},
	vol.~3, no.~1, pp. 265--284, 1892.
	
\end{thebibliography}

\end{document}